\documentclass{amsart}

\newtheorem{theorem}{Theorem}[section]
\newtheorem{lemma}{Lemma}

\newtheorem{proposition}{Proposition}

\theoremstyle{definition}
\newtheorem*{definition}{Definition}

\theoremstyle{remark}

\theoremstyle{remarks}
\newtheorem*{remarks}{Remarks}
\theoremstyle{example}

\numberwithin{equation}{section}

\begin{document}
\title{Unitary Loop Groups and Factorization}

\author{Doug Pickrell}
\email{pickrell@math.arizona.edu}
\author{Benjamin Pittman-Polletta}
\email{bpolletta@math.arizona.edu}

\begin{abstract}We discuss a refinement of triangular factorization for unitary matrix-valued functions
on $S^1$.
\end{abstract}
\maketitle

\setcounter{section}{-1}

\section{Introduction}\label{Introduction}

The purpose of this paper is to generalize Theorems
\ref{SU(2)theorem1} and \ref{U(2)theorem} below to general simply
connected compact Lie groups.

Throughout the paper $\dot K$ is a simply connected compact Lie
group, $\dot G$ is the complexification, and $L_{fin}\dot K$
($L_{fin}\dot G$) denotes the group consisting of functions $S^1
\to \dot K$ ($\dot G$, respectively) having finite Fourier series,
relative to some matrix representation, with pointwise
multiplication. For example, for $\zeta \in \mathbb C$ and $n\in
\mathbb Z$, the function
$$S^1 \to SU(2):z \to
a(\zeta) \left(\begin{matrix} 1&
\zeta z^{-n}\\
-\bar{\zeta}z^n&1\end{matrix} \right),$$ where $a(\zeta)=(1+\vert
\zeta \vert ^2)^{-1/2}$, is in $L_{fin}SU(2)$. Also, if $f(z)=\sum
f_n z^n$, then $f^*=\sum \bar f_n z^{-n}$. Thus if $f \in
H^0(\Delta)$, then $f^* \in H^0(\Delta^*)$, where $\Delta$ is the
open unit disk, and $\Delta^*$ is the open unit disk at $\infty$.

The following two theorems are from \cite{Pi2}.

\begin{theorem}\label{SU(2)theorem1} Suppose that $k_1 \in L_{fin}SU(2)$. The
following are equivalent:

($a_1$) $k_1$ is of the form
$$k_1(z)=\left(\begin{matrix} a(z)&b(z)\\
-b^*&a^*\end{matrix} \right),\quad z\in S^1,$$ where $a$ and $b$
are polynomials in $z$, and $a(0)>0$.

($b_1$) $k_1$ has a factorization of the form
$$k_1(z)=a(\eta_n)\left(\begin{matrix} 1&-\bar{\eta}_nz^n\\
\eta_nz^{-n}&1\end{matrix} \right)..a(\eta_0)\left(\begin{matrix}
1&
-\bar{\eta}_0\\
\eta_0&1\end{matrix} \right),$$ for some $\eta_j \in \mathbb C$.

($c_1$) $k_1$ has triangular factorization of the form
$$\left(\begin{matrix} 1&0\\
\sum_{j=0}^n \bar y_jz^{-j}&1\end{matrix} \right)\left(\begin{matrix} a_1&0\\
0&a_1^{-1}\end{matrix} \right)\left(\begin{matrix} \alpha_1 (z)&\beta_1 (z)\\
\gamma_1 (z)&\delta_1 (z)\end{matrix} \right),$$ where $a_1>0$ and
the third factor is a polynomial in $z$ which is unipotent upper
triangular at $z=0$.

Similarly, the following are equivalent:

($a_2$) $k_2$ is of the form
$$k_2(z)=\left(\begin{matrix} d^{*}&-c^{*}\\
c(z)&d(z)\end{matrix} \right),\quad z\in S^1,$$ where $c$ and $d$
are polynomials in $z$, $c(0)=0$, and $d(0)>0$.

($b_2$) $k_2$ has a factorization of the form
$$k_2(z)=a(\zeta_n)\left(\begin{matrix} 1&\zeta_nz^{-n}\\
-\bar{\zeta}_nz^n&1\end{matrix}
\right)..a(\zeta_1)\left(\begin{matrix} 1&
\zeta_1z^{-1}\\
-\bar{\zeta}_1z&1\end{matrix} \right),$$ for some $\zeta_j \in
\mathbb C$.

($c_2$) $k_2$ has triangular factorization of the form
$$\left(\begin{matrix} 1&\sum_{j=1}^n \bar x_jz^{-j}\\
0&1\end{matrix} \right)\left(\begin{matrix} a_2&0\\
0&a_2^{-1}\end{matrix} \right)\left(\begin{matrix} \alpha_2 (z)&\beta_2 (z)\\
\gamma_2 (z)&\delta_2 (z)\end{matrix} \right), $$ where $a_2>0$
and the third factor is a polynomial in $z$ which is unipotent
upper triangular at $z=0$.

\end{theorem}

For either $C^{\infty}$ loops, or for higher rank groups, there
are additional complications. For higher rank groups: in $(a_i)$
one has to consider multiple fundamental representations (for
$SU(2)$ there is just the defining representation); in $(b_i)$ the
factors correspond to a choice of a reduced sequence of simple
reflections in the affine Weyl group of $\dot K$ (for $SU(2)$
there is essentially one possibility); and in $(c_i)$ the first
factor is a more general unipotent matrix (for $SU(2)$ the
corresponding Lie algebra is abelian).

\begin{theorem}\label{U(2)theorem}
(a) If $\{\eta_i\}$ and $\{\zeta_j\}$ are rapidly decreasing
sequences of complex numbers, then the limits
$$k_1(z)=\lim_{n\to\infty}a(\eta_n)\left(\begin{matrix} 1&-\bar{\eta}_nz^n\\
\eta_nz^{-n}&1\end{matrix} \right)..a(\eta_0)\left(\begin{matrix}
1&
-\bar{\eta}_0\\
\eta_0&1\end{matrix} \right)$$ and
$$k_2(z)=\lim_{n\to\infty}a(\zeta_n)\left(\begin{matrix} 1&\zeta_nz^{-n}\\
-\bar{\zeta}_nz^n&1\end{matrix}
\right)..a(\zeta_1)\left(\begin{matrix} 1&
\zeta_1z^{-1}\\
-\bar{\zeta}_1z&1\end{matrix} \right),$$ exist in
$C^{\infty}(S^1,SU(2))$.

(b) Suppose $g\in C^{\infty}(S^1,SU(2))$. The following are
equivalent:

(i) $g$ has a triangular factorization $g=lmau$, where $l$ and $u$
have $C^{\infty}$ boundary values.

(ii) $g$ has a factorization of the form
$$g(z)=k_1(z)^*\left(\begin{matrix} e^{\chi}&0\\
0&e^{-\chi}\end{matrix}\right)k_2(z),$$ where $\chi \in
C^{\infty}(S^1,i\mathbb R)$, and $k_1$ and $k_2$ are as in (a).

\end{theorem}

The form of the generalization of (b) to higher rank groups is
evident.

The plan of the paper is the following. In Section \ref{notation}
we establish notation. In Section \ref{Weylgroup} we define and
prove the existence of (affine periodic) reduced sequences of Weyl
reflections, which are relevant to higher rank generalizations of
Theorem \ref{SU(2)theorem1}. In Sections \ref{theorem1} and
\ref{theorem2} we formulate and prove higher rank generalizations
of Theorems \ref{SU(2)theorem1} and \ref{U(2)theorem},
respectively. In the last section we present some examples.

In this paper our main focus is on smooth loops. It is of great
interest to generalize these factorizations to other function
spaces, and to understand how decay properties of the parameters
correspond to regularity properties of the corresponding loops.
However, even in the $SU(2)$ case, this is only partially
understood (see \cite{Pi2}, especially Section 4).

\section{Notation and Background}\label{notation}

\subsection{Finite Dimensional Algebras and Groups}

$\dot{\mathfrak k}=Lie(\dot K)$, $\dot{\mathfrak g}=\dot{\mathfrak
k}^{\mathbb C}$, and $\dot{\mathfrak g} \to \dot{\mathfrak g}:x
\to -x^*$ is the anticomplex involution fixing $\dot{\mathfrak
k}$. To simplify the exposition, we assume that $\dot{\mathfrak
k}$ is simple. Fix a triangular decomposition
\begin{equation}\label{6.1}\dot {\mathfrak g}=\dot {\mathfrak n}^{-}\oplus\dot
{\mathfrak h}\oplus\dot {\mathfrak n}^{ +}\end{equation} such that
$\dot{\mathfrak t}=\dot{\mathfrak k} \cap \dot{\mathfrak h}$ is
maximal abelian in $\dot{\mathfrak k}$; this implies
$(\dot{\mathfrak n}^+)^*=\dot{\mathfrak n}^-$. We introduce the
following standard notations: $\{\dot{\alpha_j}: 1\le j\le r\}$ is
the set of simple positive roots, $\{\dot{h_j}\}$ is the set of
simple coroots, $\{\dot{\Lambda}_j\}$ is the set of fundamental
dominant weights, $\dot{\theta}$ is the highest root, $\dot W$ is
the Weyl group, and $\langle\cdot ,\cdot\rangle$ is the unique
invariant symmetric bilinear form such that (for the dual form)
$\langle\dot{\theta} ,\dot{\theta}\rangle = 2$. For each simple
root $\gamma$, fix a root homomorphism $i_{\gamma}:sl(2,\mathbb C)
\to \dot{\mathfrak g}$ (we denote the corresponding group
homomorphism by the same symbol), and let
$$f_{\gamma}=i_{\gamma}(\left (
\begin{matrix} 0&0\\1&0
\end{matrix} \right)),\quad e_{\gamma}=i_{\gamma}(\left (
\begin{matrix} 0&1\\0&0
\end{matrix} \right)), \quad \text{and} \quad \mathbf{r}_{\gamma}=i_{\gamma}(\left (
\begin{matrix} 0&i\\i&0
\end{matrix} \right)) \in \dot T=exp(\dot{\mathfrak t});$$
$\mathbf{r}_{\gamma}$ is a representative for the simple
reflection $r_{\gamma} \in \dot W$ corresponding to $\gamma$ (we
will adhere to the convention that representatives for Weyl group
elements will be denoted by bold letters).

Introduce the lattices
$$\hat{\dot T}=\bigoplus_{1\le i\le r} \mathbb Z \dot{\Lambda}_i \quad \text{(weight
lattice)}, \quad \text{and} \quad \check{\dot T}=\bigoplus_{1\le
i\le r} \mathbb Z \dot{h}_i \qquad \text{(coroot lattice)}.$$
These lattices and bases are in duality. Recall that the kernel of
$exp:\dot{\mathfrak t} \to \dot T$ is $2 \pi i$ times the coroot
lattice. Consequently there are natural identifications
$$\hat{\dot T} \to Hom(\dot T,\mathbb T),$$ where a weight $\dot{\Lambda}$
corresponds to the character $exp(2\pi ix) \to exp(2\pi
i\dot{\Lambda} (x))$, for $x\in \dot{\mathfrak h}_{\mathbb R}$,
and
$$\check{\dot
T} \to Hom(\mathbb T,\dot T),$$ where an element $h$ of the coroot
lattice corresponds to the homomorphism $\mathbb T\to \dot
T:exp(2\pi ix)\to exp(2\pi ixh)$, for $x\in \mathbb R$. Also
$$\hat{\dot{R}}=\bigoplus_{1\le i\le r} \mathbb Z \dot{\alpha}_i \quad \text{(root lattice)}, \quad
\text{and} \quad \check{\dot R}=\bigoplus_{1\le i\le r} \mathbb Z
\dot{\Theta}_i \quad \text{(coweight lattice)}$$ where these bases
are also in duality. The $\dot{\Theta}_i$ are the fundamental
coweights.

The affine Weyl group is the semidirect product $\dot W \propto
\check{\dot T}$. For the action of $\dot W$ on $\dot{\mathfrak
h}_{\mathbb R}$, a fundamental domain is the positive Weyl chamber
$C=\{x:\dot{\alpha}_i (x)>0, i=1,..,r\}$. For the natural affine
action
\begin{equation}\label{affineaction}\dot W \propto \check{\dot T}\times \dot{\mathfrak
h}_{\mathbb R} \to \dot{\mathfrak h}_{\mathbb R}\end{equation} a
fundamental domain is the convex set
$$C_0=\{x\in C:\dot{\theta}(x)<1\} \qquad
\text{(fundamental alcove)}$$ (see page 72 of \cite{PS} for the
$SU(3)$ case). The set of extreme points for the closure of $C_0$
is $\{0\}\cup\{\frac{1}{a_i}\dot{\Theta}_i\}$, where
$\dot{\theta}=\sum a_i \dot{\alpha}_i$ (these numbers are compiled
in Section 1.1 of \cite{KW}).

Let $h_{\dot{\delta}}=\sum_{i=1}^r \dot{\Theta}_i$. Then
$2h_{\dot{\delta}}\in \check{\dot T}$. For a positive root
$\dot{\alpha}$,
$\dot{\alpha}(h_{\dot{\delta}})=height(\dot{\alpha})$.

Let $\dot N^{\pm}=exp(\dot n^{\pm})$ and $\dot A=exp(\dot
h_{\mathbb R})$. An element $g\in \dot N^-\dot T\dot A\dot N^+$
has a unique triangular decomposition
\begin{equation}\label{finitefactorization}g=\dot l(g) \dot d(g) \dot u(g), \quad \text{where} \quad
\dot d=\dot m \dot a=\prod_{j=1}^r \dot {\sigma}_j(g)^{\dot
h_j},\end{equation} and $\dot {\sigma}_i(g)=\phi_i(\pi_{\dot
{\Lambda}_i}(g) v_{\dot{\Lambda}_i})$ is the fundamental matrix
coefficient for the highest weight vector corresponding to $\dot
{\Lambda}_i$.

\subsection{Affine Lie Algebras}

Let $L\dot{\mathfrak g}=C^{\infty}(S^1,\dot{\mathfrak g})$, viewed
as a Lie algebra with pointwise bracket. There is a universal
central extension
\[0\to \mathbb C c\to\tilde {L}\dot {\mathfrak g}\to L\dot {\mathfrak g}\to 0,\]
where as a vector space $\tilde {L}\dot {\mathfrak g}=L\dot
{\mathfrak g}\oplus \mathbb C c$, and in these coordinates
\begin{equation}\label{bracket}[X+\lambda c,Y+\lambda^{\prime} c]_{\tilde {L}\dot {\mathfrak g}}
=[X,Y]_{L\dot {\mathfrak g}}+\frac i{2\pi}\int_{S^1}\langle
X\wedge dY\rangle c. \end{equation} The smooth completion of the
untwisted affine Kac-Moody Lie algebra corresponding to
$\dot{\mathfrak g}$ is
$$\hat L\dot{\mathfrak g}=\mathbb C d\propto\tilde {L}\dot {\mathfrak g}
\qquad \text{(the semidirect sum)},$$ where the derivation $d$
acts by $d(X+\lambda c)=\frac 1i\frac d{d\theta}X$, for $X\in
L\dot {\mathfrak g}$, and $[d,c]=0$. The algebra generated by
$\dot {\mathfrak k}$-valued loops induces a central extension
\[0\to i\mathbb R c\to\tilde {L}\dot {\mathfrak k}\to L\dot {\mathfrak k}\to 0 \]
and a real form $\hat L\dot{\mathfrak k}=i\mathbb R d\propto\tilde
{L}\dot {\mathfrak k}$ for $\hat L\dot{\mathfrak g}$.

We identify $\dot {\mathfrak g}$ with the constant loops in $L\dot
{\mathfrak g}$. Because the extension is trivial over $\dot
{\mathfrak g}$, there are embeddings of Lie algebras
\[\dot {\mathfrak g}\to\tilde {L}\dot {\mathfrak g}\to \hat L\dot{\mathfrak g}. \]

There are triangular decompositions
\begin{equation}\label{looptriangulardecomposition}\tilde L\dot{\mathfrak
g} =\mathfrak n^{-}\oplus \mathfrak h \oplus \mathfrak n^{ +}
\quad \text{and} \quad \hat L\dot{\mathfrak g} =\mathfrak
n^{-}\oplus (\mathbb C d+\mathfrak h) \oplus \mathfrak n^{ +},
\end{equation} where $\mathfrak h=\dot {\mathfrak h}+\mathbb C c$
and $\mathfrak n^{\pm}$ is the smooth completion of
$\dot{\mathfrak n}^{\pm} +\dot{\mathfrak g}(z^{\pm 1}\mathbb
C[z^{\pm 1}])$, respectively. The simple roots for $(\hat
L_{fin}\dot{\mathfrak g},\mathbb C d+\mathfrak h)$ are
$\{\alpha_j: 0\le j\le r\}$, where
\[\alpha_0=d^*-\dot{\theta} ,\quad\alpha_j=\dot{\alpha}_j,\quad j>0, \]
$d^*(d)=1$, $d^*(c)=0$, $d^* (\dot {\mathfrak h})=0$, and the
$\dot{\alpha}_j$ are extended to $\mathbb C d+\mathfrak h$ by
requiring $\dot{\alpha}_ j(c)=\dot{\alpha}_ j(d)=0$. The simple
coroots are $\{h_j:0\le j\le rk\dot {\mathfrak g}\}$, where
\[h_0=c-\dot {h}_{\dot{\theta}},\quad h_j=\dot {h}_j,\quad j>0.\]
For $i>0$, the root homomorphism $i_{\alpha_i}$ is
$i_{\dot{\alpha}_i}$ followed by the inclusion $\dot {\mathfrak
g}\subset \tilde L \dot{\mathfrak g}$. For $i=0$
\begin{equation}\label{roothom}i_{\alpha_0}(\left(\begin{matrix} 0&0\\
1&0\end{matrix} \right))=e_{\dot{\theta}}z^{-1},\quad
i_{\alpha_0}(\left(\begin{matrix}
0&1\\
0&0\end{matrix} \right))=f_{\dot{\theta}}z, \end{equation} where
$\{f_{\dot{\theta}},\dot {h}_{\dot{\theta}},e_{\dot{\theta}}\}$
satisfy the $ sl(2,\mathbb C )$-commutation relations, and
$e_{\dot{\theta}}$ is a highest root for $\dot {\mathfrak g}$. The
fundamental dominant integral functionals on $\mathfrak h$ are
$\Lambda_j$, $j=0,..,r$.

Also set $\mathfrak t=i\mathbb R c \oplus \dot{\mathfrak t}$ and
$\mathfrak a=\mathfrak h_{\mathbb R}=\mathbb R c \oplus
\dot{\mathfrak h}_{\mathbb R}$.

\subsection{Loop Groups and Extensions}\label{loopgroups}

Let $\Pi:\tilde L \dot G \to L \dot G$ ($\Pi:\tilde L \dot K \to L
\dot K$) denote the universal central $\mathbb C^*$ ($\mathbb T$)
extension of the smooth loop group $L \dot G$ ($L \dot K$,
respectively), as in \cite{PS}. Let $N^{\pm}$ denote the subgroups
corresponding to $\mathfrak n^{\pm}$. Since the restriction of
$\Pi$ to $N^{\pm}$ is an isomorphism, we will always identify
$N^{\pm}$ with its image, e.g. $l\in N^+$ is identified with a
smooth loop having a holomorphic extension to $\Delta$ satisfying
$l(0)\in \dot N^+$. Also set $T=exp(\mathfrak t)$ and
$A=exp(\mathfrak a)$.

As in the finite dimensional case, for $\tilde {g}\in N^{-}\cdot T
A\cdot N^{+}\subset\tilde {L}\dot G$, there is a unique triangular
decomposition
\begin{equation}\label{diagonal}\tilde {g}=l\cdot d \cdot u,\quad \text{where}\quad d
=ma=\prod_{j=0}^r\sigma_j(\tilde { g})^{h_j},\end{equation} and
$\sigma_j=\sigma_{\Lambda_j}$ is the fundamental matrix
coefficient for the highest weight vector corresponding to
$\Lambda_j$. If $\Pi(\tilde {g})=g$, then because
$\sigma_0^{h_0}=\sigma_0^{c-\dot {h}_{\dot{\theta}}}$ projects to
$\sigma_ 0^{-\dot {h}_{\dot{\theta}}}$, $g=l\cdot \Pi(d)\cdot u$,
where
\begin{equation}\label{loopdiagonal}\Pi(d)(g)=
\sigma_0(\tilde {g})^{-\dot
{h}_{\dot{\theta}}}\prod_{j=1}^r\sigma_j(\tilde {g})^{\dot
{h}_j}=\prod_{j=1}^r\left (\frac {\sigma_j(\tilde
{g})}{\sigma_0(\tilde {g})^{\check {a}_j}}\right )^{\dot {h}_j},
\end{equation} and the $\check {a}_j$ are positive integers such
that $\dot {h}_{\dot{\theta}}=\sum\check {a}_j\dot {h}_j$ (these
numbers are also compiled in Section 1.1 of \cite{KW}).

If $\tilde{g} \in \tilde{L}\dot K$, then
$\vert\sigma_j(\tilde{g})\vert$ depends only on $g=\Pi(\tilde g)$.
We will indicate this by writing
\begin{equation}\label{diagonalnotation}\vert\sigma_j(\tilde{g})\vert=\vert\sigma_j\vert(g).\end{equation}
This also implies $a(\tilde g)=a(g)$.

For later reference we summarize this discussion in the following
way.

\begin{lemma}\label{notationlemma} For $\tilde g\in \tilde L\dot K$ and $g=\Pi(\tilde g)$,
$\tilde g$ has a triangular factorization if and only if $g$ has a
triangular factorization. The restriction of the projection
$\tilde{L}\dot K \to L\dot K$ to elements with $m(\tilde g)=1$ is
injective.
\end{lemma}

\section{Reduced Sequences in the Affine Weyl Group}\label{Weylgroup}

The Weyl group $W$ for $(\hat L\dot{\mathfrak g},\mathbb
Cd+\mathfrak h)$ acts by isometries of $(\mathbb R d+\mathfrak
h_{\mathbb R},\langle \cdot,\cdot \rangle)$. The action of $W$ on
$\mathbb R c$ is trivial. The affine plane $d+\dot{\mathfrak h}$
is $W$-stable, and this action identifies $W$ with the affine Weyl
group and its affine action (\ref{affineaction}) (see Chapter 5 of
\cite{PS}). In this realization
\begin{equation}\label{r0eqn}r_{\alpha_0}= r_{\dot{\theta}} \circ
\dot h_{\dot{\theta}},\quad \text{and}\quad
r_{\alpha_i}=r_{\dot{\alpha_i}},\quad i>0.\end{equation}

\begin{definition}\label{minimal} A sequence of simple reflections
$r_1,r_2,..$ in $W$ is called reduced if $w_n=r_nr_{n-1}..r_1$ is
a reduced expression for each $n$.
\end{definition}

The following is well-known:

\begin{lemma}\label{rootlemma}Given a reduced sequence of simple
reflections $\{r_j\}$, corresponding to simple positive roots
$\gamma_j$,

(a) the positive roots which are mapped to negative roots by $w_n$
are
$$\tau_j=w_{j-1}^{-1}\cdot\gamma_j=r_1..r_{j-1}\cdot\gamma_j,\quad j=1,..,n.$$

(b) $w_{k-1}\tau_n=r_k..r_{n-1}\gamma_n>0$, $k<n$.
\end{lemma}

A reduced sequence of simple reflections determines a
non-repeating sequence of adjacent alcoves
\begin{equation}\label{alcovewalk}C_0, w_1^{-1} C_0,..,
w_{n-1}^{-1}C_0=r_1..r_{n-1}C_0,..,\end{equation} where the
step from $w_{n-1}^{-1}C_0$ to $w_n^{-1}C_0$ is implemented by the
reflection $r_{\tau_n}=w_{n-1}^{-1}r_nw_{n-1}$ (in particular the
wall between $C_{n-1}$ and $C_n$ is fixed by $r_{\tau_n}$).
Conversely, given a sequence of adjacent alcoves $(C_j)$ which is
minimal in the sense that the minimal number of steps to go from
$C_0$ to $C_j$ is $j$, there is a corresponding reduced sequence
of reflections.

\begin{definition}A reduced sequence of simple reflections
$\{r_j\}$ is affine periodic if, in terms of the identification of
$W$ with the affine Weyl group, (1) there exists $l$ such that
$w_l \in \check{\dot T}$ and (2) $w_{s+l}= w_s \circ w_l$, for all
$s$. We will refer to $w_l$ as the period ($l$ is the length of
the period).

\end{definition}

\begin{remarks} (a) The second condition is equivalent to
periodicity of the associated sequence of simple roots
$\{\gamma_j\}$, i.e. $\gamma_{s+l}=\gamma_s$.

(b) In terms of the associated walk through alcoves, affine
periodicity means that the walk from step $l+1$ onward is the
original walk translated by $w_l^{-1}$.
\end{remarks}

\begin{theorem}\label{periodic}(a) There exists an affine
periodic reduced sequence $\{r_j\}$ of simple reflections such
that, in the notation of Lemma \ref{rootlemma},
\begin{equation}\label{flips}\{\tau_j:1\le j<\infty\}=\{k d^* -\dot{\alpha
}:\dot{\alpha }>0,k=1,2,..\},\end{equation} i.e. such that the
span of the corresponding root spaces is $\dot{\mathfrak
n}^{-}(z\mathbb C[z])$. The period can be chosen to be any point
in $C \cap \check{\dot T}$.

(b) Given a reduced sequence as in (a), and a reduced expression
for $\dot {w}_0=r_{-N}..r_0$ (where $\dot w_0$ is the longest
element of $\dot W$), the sequence
$$r_{-N},..,r_0,r_1,..$$
is another reduced sequence. The corresponding set of positive
roots mapped to negative roots is
$$\{k d^* +\dot{\alpha }:\dot{\alpha }>0,k=0,1,..\}
,$$ i.e. the span of the corresponding root spaces is $\dot
{\mathfrak n}^{+}(\mathbb C[z])$.

\end{theorem}

\begin{proof} Suppose that we are given a
reduced sequence $\{r_j\}$ of simple reflections. In terms of the
corresponding sequence of adjacent alcoves (\ref{alcovewalk}),
(\ref{flips}) holds if and only if the alcoves are all contained
in $C$, and asymptotically the alcoves are infinitely far from the
walls of $C$. In the case of $SU(2)$, there is a unique such
reduced sequence; in all higher rank cases, there are infinitely
many such reduced sequences.

Now suppose that $h \in C \cap \check{\dot T}$, e.g.
$h=2h_{\dot{\delta}}$. To obtain an affine periodic reduced
sequence with period $h$, choose a finite reduced sequence of
reflections $r_1$,..,$r_d$ such that $r_1..r_d C_0=C_0+h$. This is
equivalent to doing a minimal walk through alcoves, all contained
in $C$, from $C_0$ to its translate by $h$. Note that the closure
of the alcove $C_0+h$ is contained in $C$. We then continue the
walk through alcoves periodically, that is, after the second set
of $d$ steps we land at the translate of $C_0$ by $2h$ and so on.
This periodic walk has the property that we are eventually
arbitrarily deep in the interior of $C$, and this property implies
(\ref{flips}). This implies (a).

Given (a), part (b) is clear.

\end{proof}

\begin{remarks} (a) If $h_{\dot{\delta}}\in \check{\dot T}$,
then $h_{\dot{\delta}}$ is the point in $C\cap\check{\dot T}$ with
shortest Weyl group length (and also closest to the origin in the
Euclidean sense). If this is not the case, conjecturally
$2h_{\dot{\delta}}$ is the point in $C\cap\check{\dot T}$ with
shortest Weyl group length.

(b) Via the exponential map $2\pi i \check{\dot R}/2\pi
i\check{\dot T}$ is isomorphic to $C(\dot K)$. Thus $exp(2\pi i
h_{\dot{\delta}})$ is always central, and $h_{\dot{\delta}}\in
\check{\dot T}$ if and only if $exp(2\pi i h_{\dot{\delta}})=1$.
This is the case if and only if $\dot{\mathfrak g}$ is of type
$A_l$, $l$ even, $G_2$, $F_4$, $E_6$ or $E_8$.

(c) Given a point $h\in \check{\dot T}$, the Weyl group length is
the number of hyperplanes associated to Weyl group reflections
crossed by the straight line from the origin to $h$. When
$h_{\dot{\delta}}\in \check{\dot T}$, its length is
$\sum_{\dot{\alpha}>0}height(\dot{\alpha})$.
\end{remarks}

\section{Generalizations of Theorem \ref{SU(2)theorem1}}\label{theorem1}

Throughout this section we assume that we have chosen a reduced
sequence $\{r_j\}$ as in Theorem \ref{periodic}, and a reduced
expression $\dot w_0=r_{-N}..r_0$ (In Theorem \ref{Ktheorem1}
below, it is not necessary to assume that the sequence is affine
periodic. We will use affine periodicity in Theorem
\ref{Ktheorem1smooth}, and it seems plausible that this use is
essential). We set
$$i_{\tau_n}=\mathbf w_{n-1} i_{\gamma_n} \mathbf
w_{n-1}^{-1}, \qquad n=1,2,..$$
$$i_{\tau_{-N}^{\prime}}=i_{\gamma_{-N}},\quad
i_{\tau_{-(N-1)}^{\prime}}=\mathbf
r_{-N}i_{\gamma_{-(N-1)}}\mathbf r_{-N}^{-1},..,\quad
i_{\tau_0^{\prime}}=\dot{\mathbf w}_0 i_{\gamma_0} \dot{\mathbf
w}_0^{-1}
$$ and for $n>0$
$$i_{\tau_n^{\prime}}=\dot{\mathbf w}_0 \mathbf w_{n-1}
i_{\gamma_n} \mathbf w_{n-1}^{-1} \dot{\mathbf w}_0^{-1}.$$

Also for $\zeta\in \mathbb C$, let
$a(\zeta)=(1+\vert\zeta\vert^2)^{-1/2}$ and
\begin{equation}\label{kfactor}k(\zeta)=a(\zeta)\left (
\begin{matrix}1&-\bar{\zeta}\\\zeta&1
\end{matrix}\right)=\left ( \begin{matrix}1&0\\\zeta&1
\end{matrix}\right)\left ( \begin{matrix}a(\zeta)&0\\0&a(\zeta)^{-1}
\end{matrix}\right)\left ( \begin{matrix}1&-\bar{\zeta}\\0&1
\end{matrix}\right) \in SU(2).\end{equation}

\begin{theorem}\label{Ktheorem1} Suppose that $\tilde{k}_1 \in \tilde{L}_{fin}\dot{K}$
and $\Pi(\tilde{k}_1)=k_1$. The following are equivalent:

($a_1$) $m(\tilde{k}_1)=1$; and for each complex irreducible
representation $V(\pi)$ for $\dot{G}$, with lowest weight vector
$\phi \in V(\pi)$, $\pi(k_1)^{-1}(\phi)$ is a polynomial in $z$
(with values in $V$), and is a positive multiple of $v$ at $z=0$.

($b_1$) $\tilde{k}_1$ has a factorization of the form
$$\tilde{k}_1=i_{\tau'_n}(k(\eta_n))..i_{\tau'_{-N}}(k(\eta_{-N}))\in \tilde{L}_{fin}\dot K$$
for some $\eta_j \in \mathbb C$.

($c_1$) $\tilde{k}_1$ has triangular factorization of the form
$\tilde k_1=l_1a_1u_1$ where $l_1\in \dot N^-(\mathbb C[z^{-1}])$.

Moreover, in the notation of ($b_1$),
$$a_1=\prod a(\eta_j)^{h_{\tau^{\prime}_j}}.$$

Similarly, the following are equivalent:

($a_2$) $m(\tilde k_2)=1$, and for each complex irreducible
representation $V(\pi)$ for $\dot{G}$, with highest weight vector
$v \in V(\pi)$, $\pi(k_2)^{-1}(v)$ is a polynomial in $z$ (with
values in $V$), and is a positive multiple of $v$ at $z=0$.

($b_2$) $\tilde k_2$ has a factorization of the form
$$\tilde k_2=i_{\tau_n}(k(\zeta_n))..i_{\tau_1}(k(\zeta_1))$$
for some $\zeta_j \in \mathbb C$.

($c_2$) $\tilde k_2$ has triangular factorization of the form
$\tilde k_2=l_2a_2u_2$, where $l_2 \in \dot N^+(z^{-1}\mathbb C
[z^{-1}])$.

Also, in the notation of ($b_2$),
\begin{equation}\label{productformula}a_2=\prod a(\zeta_j)^{h_{\tau_j}}.\end{equation}

\end{theorem}

\begin{proof} The two sets of equivalences are proven in the same way.
We consider the second set.

The subalgebra $\mathfrak n^- \cap \mathbf w_{n-1}^{-1} \mathfrak
n^+ \mathbf w_{n-1}$ is spanned by the root spaces corresponding
to negative roots $-\tau_j$, $j=1,..,n$. Given this, the
equivalence of ($b_2$) and ($c_2$) follows from (b) of Proposition
3 of \cite{Pi1}. We recall the argument, because we will need to
refine it (This refinement is carried out in the Appendix, since
it is rather technical and not relevant for present purposes). In
the process we will also recall the proof of the product formula
for $a_2$.

The equation (\ref{kfactor}) implies that
$$i_{\tau_j}(k(\zeta_j))=i_{\tau_j}(
\left ( \begin{matrix}1&0\\\zeta_j&1
\end{matrix}\right)) a(\zeta_j)^{h_{\tau_j}}i_{\tau_j}(
\left ( \begin{matrix}1&-\bar{\zeta_j}\\0&1
\end{matrix}\right))$$
$$=exp(\zeta_j
f_{\tau_j})a(\zeta_j)^{h_{\tau_j}}\mathbf
w_{j-1}^{-1}exp(-\bar{\zeta}_j e_{\gamma_j})\mathbf w_{j-1}$$ is a
triangular factorization.

Let $k^{(n)}=i_{\tau_n}(k(\zeta_n))..i_{\tau_1}(k(\zeta_1))$.
First suppose that $n=2$. Then
\begin{equation}\label{n=2case}k^{(2)}=
exp(\zeta_2 f_{\tau_2}) a(\zeta_2)^{h_{\tau_2}}\mathbf
r_1exp(-\bar{\zeta}_2 e_{\gamma_2})\mathbf r_1^{-1} exp(\zeta_1
f_{\gamma_1})a(\zeta_1)^{h_{\gamma_1}}exp(-\bar{\zeta}_1
e_{\gamma_1})\end{equation} The key point is that
$$\mathbf r_1exp(-\bar{\zeta}_2 e_{\gamma_2})\mathbf r_1^{-1} exp(\zeta_1
f_{\gamma_1})=\mathbf r_1 exp(-\bar{\zeta}_2 e_{\gamma_2})
exp(\zeta_1 e_{\gamma_1})\mathbf r_1^{-1}$$
$$=\mathbf r_1 exp(\zeta_1
e_{\gamma_1})\tilde u \mathbf r_1^{-1},\quad (\text{for some}
\quad \tilde u\in N^+\cap r_1 N^+ r_1^{-1})$$
$$= exp(\zeta_1
f_{\gamma_1})\mathbf u, \quad (\text{for some} \quad \mathbf u\in
N^+).$$ Insert this calculation into (\ref{n=2case}). We then see
that $k^{(2)}$ has a triangular factorization, where
$$a(k^{(2)})=
a(\zeta_1)^{h_{\tau_1}}a(\zeta_2)^{h_{\tau_2}}$$ and
\begin{equation}\label{2ndcase}l(k^{(2)})=exp(\zeta_2
f_{\tau_2})exp(\zeta_1a(\zeta_2)^{-\tau_1(h_{\tau_2})}f_{\tau_1})\end{equation}
$$=exp(\zeta_2
f_{\tau_2}+\zeta_1a(\zeta_2)^{-\tau_1(h_{\tau_2})}f_{\tau_1})$$
(the last equality holds because a two dimensional nilpotent
algebra is necessarily commutative).

To apply induction, we assume that $k^{(n-1)}$ has a triangular
factorization with
\begin{equation}\label{induction}l(k^{(n-1)})=exp(\zeta_{n-1}
f_{\tau_{n-1}})\tilde l \in N^-\cap
w_{n-1}^{-1}N^+w_{n-1}=exp(\sum_{j=1}^{n-1}\mathbb C
f_{\tau_j}),\end{equation} for some $\tilde l \in N^-\cap
w_{n-2}^{-1}N^+w_{n-2}=exp(\sum_{j=1}^{n-2}\mathbb C f_{\tau_j})$,
and
$$a(k^{(n-1)})= \prod_{j=1}^{n-1}a(\zeta_j)^{h_{\tau_j}}.$$
We have established this for $n-1=1,2$. For $n \ge 2$
$$k^{(n)}=exp(\zeta_n f_{\tau_n})a(\zeta_n)^{h_{\tau_n}}\mathbf w_{n-1}^{-1}
exp(-\bar{\zeta}_n e_{\gamma_n})\mathbf w_{n-1}exp(\zeta_{n-1}
f_{\tau_{n-1}})\tilde l a(k^{(n-1)})u(k^{(n-1)})$$

$$=exp(\zeta_n f_{\tau_n})a(\zeta_n)^{h_{\tau_n}}\mathbf w_{n-1}^{-1}
exp(-\bar{\zeta}_ne_{\gamma_n}) \tilde u\mathbf
w_{n-1}a(k^{(n-1)})u(k^{(n-1)}),$$ where $\tilde u= \mathbf
w_{n-1}exp(\zeta_{n-1} f_{\tau_{n-1}})\tilde l \mathbf
w_{n-1}^{-1}\in \mathbf w_{n-1} N^{-} \mathbf w_{n-1}^{-1} \cap
N^+$. Now write $exp(-\bar{\zeta}_ne_{\gamma_n}) \tilde u=\tilde
u_1\tilde u_2$, relative to the decomposition
$$N^+=\left(N^+\cap w_{n-1}N^-w_{n-1}^{-1}\right)\left(N^+\cap
w_{n-1}N^+w_{n-1}^{-1}\right).$$ Let
$$\mathbf l=a(\zeta_n)^{h_{\tau_n}}\mathbf
w_{n-1}^{-1}\tilde u_1 \mathbf w_{n-1}a(\zeta_n)^{-h_{\tau_n}}\in
N^- \cap \mathbf w_{n-1}^{-1} N^+ \mathbf w_{n-1}.$$ Then
$k^{(n)}$ has triangular decomposition
$$k^{(n)}=\left(exp(\zeta_n f_{\tau_n})\mathbf
l\right) \left(a(\zeta_n)^{h_{\tau_n}}a(k^{(n-1)})\right)\left(
a(k^{(n-1)})^{-1}\tilde u_2a(k^{(n-1)})u(k^{(n-1)})\right).$$ This
implies the induction step. Because of the form of
(\ref{induction}), it is clear that $\zeta_1,..,\zeta_n$ are
global coordinates for $N^-\cap w_{n}^{-1}N^+w_{n}$. This
establishes the equivalence of ($b_2$) and ($c_2$). The induction
statement also implies the product formula (\ref{productformula})
for $a_2$.

It is obvious that ($c_2$) implies ($a_2$). In fact ($c_2$)
implies a stronger condition. If ($c_2$) holds, then given a
highest weight vector $v$ as in ($a_2$), corresponding to highest
weight $\dot{\Lambda}$, then
\begin{equation}\label{highestweight}
\pi(k_2^{-1})v=\pi(u_2^{-1}a_2^{-1}l^{-1})v=a_2^{-\dot{\Lambda}}\pi(u_2^{-1})v,\end{equation}
implying that $\pi(k_2^{-1})v$ is holomorphic in $\Delta$ and
nonvanishing at all points. However we do not need to include this
nonvanishing condition in ($a_2$), in this finite case.

It remains to prove that ($a_2$) implies ($c_2$). Because $\tilde
k_2$ is determined by $k_2$, as in Lemma \ref{notationlemma}, it
suffices to show that $k_2$ has a triangular factorization (with
trivial $\dot T$ component). Hence we will slightly abuse notation
and work at the level of loops in the remainder of this proof.

To motivate the argument, suppose that $k_2$ has triangular
factorization as in ($c_2$). Because $u_2(0)\in \dot{N}^+$, there
exists a pointwise $\dot{G}$-triangular factorization (see
(\ref{finitefactorization}))
\begin{equation}\label{pointwise}u_2(z)^{-1}=
\dot{l}(u_2(z)^{-1})\dot{d}(u_2(z)^{-1})\dot{u}(u_2(z)^{-1})\end{equation}
which is certainly valid in a neighborhood of $z=0$; more
precisely, (\ref{pointwise}) exists at a point $z\in \mathbb C$ if
and only if
$$\dot{\sigma}_i(u_2(z)^{-1})\ne 0, \quad i=1,..,r.$$
When (\ref{pointwise}) exists (and using the fact that $k_2$ is
defined in $\mathbb C^*$),
$$k_2(z)=\left (l_2(z)a_2 \dot{u}(u_2(z)^{-1})^{-1} a_2^{-1} \right)
\left ( a_2 \dot{d}(u_2(z)^{-1})^{-1} \right )
\dot{l}(u_2(z)^{-1})^{-1}.$$ This implies
\begin{equation}\label{k2pointwise}k_2(z)^{-1}=\dot{l}(u_2(z)^{-1})\left (
\dot{d}(u_2(z)^{-1})a_2^{-1} \right ) \left (a_2
\dot{u}(u_2(z)^{-1}) a_2^{-1} l_2(z)^{-1}\right) .\end{equation}
This is a pointwise $\dot{G}$-triangular factorization of
$k_2^{-1}$, which is certainly valid in a punctured neighborhood
of $z=0$. The important facts are that (1) the first factor in
(\ref{k2pointwise})
\begin{equation}\label{removable}\dot{l}(k_2^{-1})=\dot{l}(u_2(z)^{-1})\end{equation}
does not have a pole at $z=0$; (2) for the third (upper
triangular) factor in (\ref{k2pointwise}), the factorization
\begin{equation}\label{unobstructed}\dot{u}(k_2^{-1}) ^{-1}= l_2(z) \left
(a_2\dot{u}(u_2(z)^{-1})a_2^{-1} \right )\end{equation} is a
$L\dot{G}$-triangular factorization of $\dot{u}(k_2^{-1})^{-1}\in
L\dot{N}^+$, where we view $\dot{u}(k_2^{-1})^{-1}$ as a loop by
restricting to a small circle surrounding $z=0$; and (3) because
there is an a priori formula for $a_2$ in terms of $k_2$ (see
(\ref{loopdiagonal}), we can recover $l_2$ and (the pointwise
triangular factorization for) $u_2^{-1}$ from
(\ref{k2pointwise})-(\ref{unobstructed}): $l_2=l(\dot
u(k_2^{-1})^{-1})$ (by (\ref{unobstructed})), and
\begin{equation}\label{backwards}\dot l(u_2(z)^{-1})=\dot l( k_2(z)^{-1}),
\quad \dot d(u_2(z)^{-1})=\dot d(k_2(z)^{-1})a_2,
\end{equation}
$$ \text{and} \quad \dot u(u_2(z)^{-1})=a_2^{-1}u(\dot
u(k_2(z)^{-1}))a_2.$$ We remark that this uses the fact that $k_2$
is defined in $\mathbb C^*$ in an essential way.

Now suppose that ($a_2$) holds. In particular ($a_2$) implies that
$\dot{\sigma}_i(k_2^{-1})$ has a removable singularity at $z=0$
and is positive at $z=0$, for $i=1,..,r$. Thus $k_2^{-1}$ has a
pointwise $\dot{G}$-triangular factorization as in
(\ref{k2pointwise}), for all $z$ in some punctured neighborhood of
$z=0$.

We claim that (\ref{removable}) does not have at pole at $z=0$. To
see this, recall that for an $n\times n$ matrix $g=(g_{ij})$
having an LDU factorization, the entries of the factors can be
written explicitly as ratios of determinants:
$$\dot{d}(g)=diag(\sigma_1,\sigma_2/\sigma_1,\sigma_3/\sigma_2,..,\sigma_
n/\sigma_{n-1})$$ where $\sigma_k$ is the determinant of the
$k^{th}$ principal submatrix, $\sigma_k=det((g_{ij})_{1\le i,j\le
k})$; for $i>j$,
\begin{equation}\label{lower}l_{ij}=det\left(\begin{matrix} g_{11}&g_{12}&..&g_{1j}\\
g_{21}&&&\\
.\\
.\\
g_{j-1,1}&&&g_{j-1,j}\\
g_{i,1}&&&g_{ij}\end{matrix} \right)/\sigma_j=\frac {\langle
g\epsilon_ 1\wedge ..\wedge\epsilon_j,\epsilon_1\wedge
..\wedge\epsilon_{j-1} \wedge\epsilon_i\rangle}{\langle
g\epsilon_1\wedge ..\wedge\epsilon_ j,\epsilon_1\wedge
..\wedge\epsilon_j\rangle}\end{equation}
and for $i<j$,
$$u_{ij}=det\left(\begin{matrix} g_{11}&g_{12}&..&g_{1,i-1}&g_{1,j}\\
.&&&&g_{2,j}\\
.\\
g_{i,1}&&&&g_{i,j}\end{matrix} \right)/\sigma_i.$$ Apply this to
$g=k_2^{-1}$ in a highest weight representation. Then
(\ref{lower}), together with ($a_2$), implies the claim.

The factorization (\ref{unobstructed}) is unobstructed. Thus it
exists. We can now read the calculation backwards, as in
(\ref{backwards}), and obtain a triangular factorization for $k_2$
as in ($c_2$) (initially for the restriction to a small circle
about $0$; but because $k_2$ is of finite type, this is valid also
for the standard circle). This completes the proof.

\end{proof}

In the $C^{\infty}$ analogue of Theorem \ref{Ktheorem1}, it is
necessary to add further hypotheses in ($a_i$); see
(\ref{highestweight}). To reiterate, we are now assuming that the
sequence $\{r_j\}$ is affine periodic.

\begin{theorem}\label{Ktheorem1smooth} Suppose that $\tilde k_1 \in \tilde L \dot{K}$
and $\Pi(\tilde k_1)=k_1$. The following are equivalent:

($a_1$) $m(\tilde k_1)=1$; and for each complex irreducible
representation $V(\pi)$ for $\dot{G}$, with lowest weight vector
$\phi \in V(\pi)$, $\pi(k_1)^{-1}(\phi)$ has holomorphic extension
to $\Delta$, is nonzero at all $z\in \Delta$, and is a positive
multiple of $v$ at $z=0$.

($b_1$) $\tilde k_1$ has a factorization of the form
$$\tilde k_1=\lim_{n\to\infty}i_{\tau'_n}(k(\eta_n))..i_{\tau'_{-N}}(k(\eta_{-N})),$$
for a rapidly decreasing sequence $(\eta_j)$.

($c_1$) $\tilde k_1$ has triangular factorization of the form
$\tilde k_1=l_1a_1u_1$ where $l_1\in H^0(\Delta^*,\dot N^-)$ has
smooth boundary values.

Moreover, in the notation of ($b_1$),
$$a_1=\prod a(\eta_j)^{h_{\tau^{\prime}_j}}.$$

Similarly, the following are equivalent: for $\tilde k_2 \in
\tilde L \dot{K}$,

($a_2$) $m(\tilde k_2)=1$; and for each complex irreducible
representation $V(\pi)$ for $\dot{G}$, with highest weight vector
$v \in V(\pi)$, $\pi(k_2)^{-1}(v)\in H^0(\Delta;V)$ has
holomorphic extension to $\Delta$, is nonzero at all $z\in
\Delta$, and is a positive multiple of $v$ at $z=0$.

($b_2$) $\tilde k_2$ has a factorization of the form
$$\tilde k_2=\lim_{n\to\infty}i_{\tau_n}(k(\zeta_n))..i_{\tau_1}(k(\zeta_1))$$
for some rapidly decreasing sequence $(\zeta_j)$.

($c_2$) $\tilde k_2$ has triangular factorization of the form
$\tilde k_2=l_2a_2u_2$, where $l_2 \in H^0(\Delta^*,\infty;\dot
N^+,1)$ has smooth boundary values.

Also, in the notation of ($b_2$),
\begin{equation}\label{product2}a_2=\prod a(\zeta_j)^{h_{\tau_j}}.\end{equation}

\end{theorem}

\begin{proof} The two sets of equivalences are proven in the same way.
We consider the second set.

Suppose that ($a_2$) holds. To show that ($c_2$) holds, it
suffices to prove that $k_2$ has a triangular factorization with
$l_2$ of the prescribed form. By working in a fixed faithful
highest weight representation for $\dot{\mathfrak g}$, without
loss of generality, we can suppose $\dot K=SU(n)$. For the
purposes of this proof, we will use the terminology in Section 1
of \cite{Pi2}. We view $k_2\in LSU(n)$ as a unitary multiplication
operator on the Hilbert space $\mathcal H=L^2(S^1;\mathbb C^n)$,
and we write
$$M_{k_2}=\left(\begin{matrix}A(k_2)&B(k_2)\\C(k_2)&D(k_2)\end{matrix}\right)$$
relative to the Hardy polarization $\mathcal H=\mathcal H^+\oplus
\mathcal H^-$, where $A(k_2)$ is the compression of $M_{k_2}$ to
$\mathcal H^+$, the subspace of functions in $\mathcal H$ with
holomorphic extension to $\Delta$. To show that $k_2$ has a
Birkhoff factorization, we must show that $A(k_2)$ is invertible
(see Theorem 1.1 of \cite{Pi2}). An elementary argument then shows
that $k_2$ has a triangular factorization of the desired form.

Let $C_1,..,C_n$ denote the columns of $k_2^{-1}$. In particular
($a_2$) implies that $C_1$ has holomorphic extension to $\Delta$
and $C_n$ has holomorphic extension to $\Delta^*$ (by considering
the dual representation). Now suppose that $f\in \mathcal H^+$ is
in the kernel of $A(k_2)$. Then
\begin{equation}\label{keyfact}(C_j^*f)_+=0, \quad j=1,..,n,\end{equation}
where $(\cdot)_+$ denotes orthogonal projection to $\mathcal H^+$,
and $C_j^*$ is the (pointwise) Hermitian transpose. Since $C_n^*$
has holomorphic extension to $\Delta$, $(C_n^*f)_+=C_n^*f$ is
identically zero on $S^1$. This implies that for $z\in S^1$,
$f(z)$ is a linear combination of the $n-1$ columns $C_j(z)$,
$j<n$. We write
$$f=\lambda_1C_1+..+\lambda_{n-1} C_{n-1}$$
where the coefficients are functions on the circle (defined a.e.).
Now consider the pointwise wedge product of $\mathbb C^n$ vectors
$$f\wedge C_1\wedge..\wedge C_{n-2}=\pm
\lambda_{n-1}C_1\wedge..\wedge C_{n-1}.$$ The vectors
$C_1\wedge..\wedge C_{j}$ extend holomorphically to $\Delta$, and
never vanish, for any $j$, by ($a_2$) (by considering the
representation $\Lambda^j(\mathbb C^n)$). Since $f$ also extends
holomorphically, this implies that $\lambda_{n-1}$ has holomorphic
extension to $\Delta$. Now
$$C_{n-1}^*f=\lambda_{n-1}C_{n-1}^*C_{n-1}=\lambda_{n-1}$$ by
pointwise orthonormality of the columns. Since the right hand side
is holomorphic in $\Delta$, by (\ref{keyfact}) (for $j=n-1$)
$\lambda_{n-1}$ vanishes identically. This implies that in fact
$f$ is a (pointwise) linear combination of the first $n-2$ columns
of $k_2^{-1}$. Continuing the argument in the obvious way (by next
wedging $f$ with $C_1\wedge..\wedge C_{n-3}$ to conclude that
$\lambda_{n-2}$ must vanish), we conclude that $f$ is zero. This
implies that $ker(A(k_2))=0$. Since $\dot K$ is simply connected,
$A(k_2)$ has index zero. Hence $A(k_2)$ is invertible. This
implies ($c_2$).

It is obvious that ($c_2$) $\implies$ ($a_2$); see
(\ref{highestweight}). Thus ($a_2$) and ($c_2$) are equivalent.

Before showing that ($b_2$) is equivalent to ($a_2$) and ($c_2$),
we need to explain why the $C^{\infty}$ limit in ($b_2$) exists.
Because $k(\zeta_j)=1+O(\vert\zeta_j\vert)$ as $\zeta_j \to 0$,
the condition for the product in ($b_2$) to converge absolutely is
that $\sum \zeta_n$ converges absolutely. So $k_2$ certainly
represents a continuous loop. Let
$$\tilde k_2^{(l)}(\zeta_1,..,\zeta_l)=i_{\tau_l}(k(\zeta_l))..i_{\tau_1}(k(\zeta_1)),$$
where $l$ is the length of the period for our affine periodic
sequence of simple reflections. Because $\tau_{l+1}=w_l^{-1}\cdot
\tau_1$, $\tau_{l+2}=w_l^{-1}\tau_2$, and in general
$\tau_{nl+j}=w_l^{-n}\cdot \tau_j$, the limit in ($b_2$) can be
written as
\begin{equation}\label{affineeqn}\tilde
k_2=\lim_{n\to\infty}\left(\mathbf w_l^{-n}\tilde
k_2^{(l)}(Z_{n})\mathbf w_l^{n} ..\tilde
k_2^{(l)}(Z_0)\right),\end{equation} where
$Z_n=(\zeta_{nl+1},..,\zeta_{nl+l})$. Now suppose that we fix a
matrix representation for $\dot K$. The unitary loop $k_2^{(l)}$,
for any argument, has a fixed order, i.e. number of Fourier
coefficients. Also $\mathbf w_l\in Hom(S^1,\dot T)$ is fixed. This
implies the order of $\mathbf w_l^{-n}k_2^{(l)}(Z_{n})\mathbf
w_l^{n}$ is asymtotically $n$. Thus when one differentiates this
term, coefficients on the order of $n$ will appear. This implies
that if $\zeta$ behaves like $n^{-p}$, then $k_2$ will have
roughly $p$ derivatives. In particular if $\zeta\in c^{\infty}$,
the Frechet space of rapidly decreasing sequences, then $k_2\in
C^{\infty}$.

Now suppose that ($b_2$) holds. The map from $\zeta$ to $\tilde
k_2$ is continuous, with respect to the standard Frechet
topologies for rapidly decreasing sequences and smooth functions.
The product (\ref{product2}) is also a continuous function of
$\zeta$, and hence is nonzero. This implies that $\tilde k_2$ has
a triangular factorization which is the limit of triangular
factorizations of the terms on the right hand side of
(\ref{affineeqn}). By Theorem \ref{Ktheorem1} and continuity, this
factorization will have the special form in ($c_2$). Thus ($b_2$)
$\implies$ ($a_2$) and ($c_2$).

Suppose that we are given $k_2$ as in ($a_2$) and ($c_2$). Recall
that $l_2$ has values in $\dot N^+$. We can take the logarithm,
and this will have an expansion
\begin{equation}\label{x*defn}log(l_2)=\sum_{j=1}^{\infty}x_j^*f_{\tau_j},\quad
x_j^*\in\mathbb C\end{equation} (the use of $x^*$ for the
coefficients is consistent with our notation in the $SU(2)$ case,
see ($c_2$) of Theorem \ref{SU(2)theorem1}). In the notation of
the proof of Theorem \ref{Ktheorem1}, the corresponding truncated
sum
$$log(l_2^{(N)})=\sum_{j=1}^{N}x_j^*f_{\tau_j}\in
\mathfrak n^- \cap w_{N-1}^{-1}\mathfrak n^+ w_{N-1}.$$ By Theorem
\ref{Ktheorem1} there is a corresponding $k_2^{(N)} \in
L_{fin}\dot K$ and a uniquely determined sequence
$(\zeta_1^{(N)},..,\zeta_N^{(N)})$ such that $k_2^{(N)}$ is
represented as in ($b_2$) and $k_2$ is a $C^{\infty}$ limit of
loops $k_2^{(N)}$. Since the $k_2^{(N)}$ converge to $k_2$, the
corresponding products (\ref{product2}) will converge. In
particular $a_2^{\Lambda_0}=\prod
(1+\vert\zeta_j\vert^2)^{-\Lambda_0(h_{\tau_j})}$ will converge.
Since $\Lambda_0(h_{\tau_j})$ is asymptotically $j$, this implies
that the sequence $\zeta^{(N)}$ is uniformly bounded in $w^{1/2}$.
Consequently this sequence of sequences will have a convergent
subsequence $\zeta\in l^2$. Thus we can assume that
$$\tilde k_2^{(N)}=i_{\tau_N}(k(\zeta_N))..i_{\tau_1}(k(\zeta_1)).$$
We can now reverse the reasoning of the previous paragraph.
Because the product converges in $C^{\infty}$, the product of the
$i_{\tau_N}(k(\zeta_N))$ converges absolutely, implying $\sum
\zeta_n$ is absolutely summable.

Finally (\ref{product2}) follows by continuity from
(\ref{productformula}).

\end{proof}

\section{Generalization of Theorem
\ref{U(2)theorem}}\label{theorem2}

\begin{theorem}\label{smooththeorem}Suppose $\tilde g \in \tilde L\dot K$ and $\Pi(\tilde g)=g$.

(a) The following are equivalent:

(i) $\tilde g$ has a triangular factorization $\tilde g=lmau$,
where $l$ and $u$ have $C^{\infty}$ boundary values.

(ii) $\tilde g$ has a factorization of the form
$$\tilde g=\tilde k_1^* exp(\chi)\tilde k_2,$$ where $\chi \in \tilde L\dot{\mathfrak t}$,
and $\tilde k_1$ and $\tilde k_2$ are as in Theorem
\ref{Ktheorem1smooth}.

(b) In reference to part (a),
\begin{equation}\label{diagonalclaim}a(\tilde g)=a(g)=
a(k_1)a(exp(\chi))a(k_2),\quad
\Pi(a(g))=\Pi(a(k_1))\Pi(a(k_2))\end{equation} and
\begin{equation}\label{abeliandiagonal}a(exp(\chi))=\vert \sigma_0
\vert(exp(\chi))^{h_0}\prod_{j=1}^r \vert \sigma_0
\vert(exp(\chi))^{\check a_j h_j}.  \end{equation}
\end{theorem}

\begin{proof}  It suffices to prove (a) with $g$ in place of
$\tilde g$, $k_1$ in place of $\tilde k_1$, and so on. This will
require that we use the clumsy (but correct) notation $k_i=l_i
\Pi(a_i)u_i$ for the triangular decomposition of $k_i$. We can
also assume that $\chi \in L\dot{\mathfrak t}$.

Suppose that we are given $g$ as in (ii). Both $k_1$ and $k_2$
have triangular factorizations. In the notation of Theorem
\ref{Ktheorem1smooth}
$$g=(l_1\Pi(a_1)u_1)^*exp(\chi)(l_2\Pi(a_2)u_2)=u_1^* \Pi(a_1)
(l_1^*exp(\chi)l_2)\Pi(a_2)u_2.$$ The simple observation is that
$b=l_1^*exp(\chi)l_2\in C^{\infty}(S^1,\dot B^+)_0$ (the identity
component), and hence will have a triangular factorization. More
precisely, if we write $\chi=\chi_-+\chi_0+\chi_+$, where
$\chi_-\in H^0(\Delta^*,\infty;\dot{\mathfrak h},0)$, $\chi_0\in
\dot{\mathfrak t}$, and $\chi_+\in H^0(\Delta,0;\dot{\mathfrak
h},0)$, then
$$b=exp(\chi_-)\left(exp(-\chi_-)l_1^*exp(\chi_-)exp(\chi_0)exp(\chi_+)l_2exp(-\chi_+)\right)exp(\chi_+)$$
will have triangular factorization
$$=\left(exp(\chi_-)L\right)\left(m(b)a(b)\right)\left(Uexp(\chi_+)\right),$$
where $m(b)=exp(\chi_0)$, $a(b)=1$,
$$L=l(exp(-\chi_-)l_1^*exp(\chi_-)exp(\chi_0)exp(\chi_+)l_2exp(-\chi_+))\in H^0(\Delta^*,\infty;\dot
N^+,1),$$ and
$$U=u(exp(-\chi_-)l_1^*exp(\chi_-)exp(\chi_0)exp(\chi_+)l_2exp(-\chi_+))\in H^0(\Delta;\dot
N^+).$$ Thus $g$ will have a triangular factorization with
\begin{equation}\label{gfactorization}l(g)=u_1^*exp(\chi_-)\Pi(a_1)L\Pi(a_1)^{-1},\quad
\Pi(m(g))=exp(\chi_0),\end{equation}
$$\Pi(a(g))=\Pi(a_1)\Pi(a_2),\quad
u(g)=\Pi(a_2)^{-1}U\Pi(a_2)exp(\chi_+)u_2.$$ Thus (ii) $\implies$
(i).

Conversely, suppose $g=l\Pi(ma)u$, as in (i). At each point of the
circle there are $\dot N^+ \dot A \dot K$ decompositions
$$l^{-1}=\dot n_1^{-1} \dot a_1^{-1} \dot k_1,\quad u=\dot n_2
\dot a_2 \dot k_2.$$ In turn there are Birkhoff decompositions
$$\dot a_i=exp(\chi_i^*+\chi_{i,0}+\chi_i),
\quad \chi_i\in H^0(\Delta,\dot h),\quad \chi_{i,0}\in \dot
h_{\mathbb R}$$ Define
$$k_1=exp(-\chi_1^*+\chi_1)\dot k_1,\quad \text{and} \quad  k_2=exp(-\chi_2^*+\chi_2)\dot k_2$$
Then
$$k_1=exp(\chi_{1,0}+2\chi_1)\dot n_1^{-*}l^*$$ has
triangular factorization with
$$l(k_1)=l(
exp(\chi_{1,0}+2\chi_1)\dot n_1^{-*}exp(-(\chi_{1,0}+2\chi_1)))\in
H^0(\Delta^*,\infty;\dot N^-,1), \Pi(a(k_1))=exp(\chi_{1,0}),$$
and similarly
$$k_2=exp(-2\chi_2^*-\chi_{2,0})\dot n_2^{-1}u$$ has
triangular factorization with
$$l(k_2)=l(
exp(-2\chi_2^*-\chi_{2,0})\dot
n_2^{-1}exp(2\chi_2^*+\chi_{2,0}))\in H^0(\Delta^*,\infty;\dot
N^+,1), \Pi(a(k_2))=exp(\chi_{2,0})$$

Finally, and somewhat miraculously, on the one hand $k_1 g
k_2^{-1} $ has value in $\dot K$, and on the other hand
\begin{equation}\label{uppert}k_1 g k_2^{-1}=exp(\chi_{1,0}+2\chi_1^*)\dot n_1 ma\dot
n_2 exp(\chi_{2,0}+2\chi_2),\end{equation}  has values in $\dot
B^+$. Therefore $k_1 g k_2^{-1}$ has values in $\dot T$. It is
also clear that (\ref{uppert}) is connected to the identity, and
hence $k_1 g k_2^{-1}\in (L\dot T)_0$. Thus (i) $\implies$ (ii).

Part (b) follows from (\ref{gfactorization}) and
(\ref{loopdiagonal}).

\end{proof}

\section{Examples}\label{examples}

In the case of $\dot K=SU(n)$, the fundamental matrix coefficient
can be realized as the determinant of the Toeplitz operator $A(g)$
associated to a loop $g\in LSU(n)$, where the determinant is
viewed as a section of a line bundle; see p. 226 of \cite{PS}.
Part (b) of Theorem \ref{smooththeorem} implies that
$$\vert \sigma_0 \vert^2(g)=det(A(g)^*A(g))$$
$$=\prod_{j=1}^{\infty}(1+\vert \eta_j
\vert^2)^{-\Lambda_0(h_{\tau_j})}exp(-n\sum k\vert
\chi_k\vert^2)\prod_{i=1}^{\infty}(1+\vert
\zeta_i\vert^2)^{-\Lambda_0(h_{\tau_i^{\prime}})}.$$ If
$\tau_j=k(j)d^*-\dot{\alpha}$ (as in Theorem \ref{periodic}), then
$\Lambda_0(h_{\tau_j})=k(j)$.

\subsection{SU(2)}

In this case there is a unique reduced sequence of simple
reflections, as in Theorem \ref{periodic}, corresponding to the
periodic sequence of simple roots
$\alpha_0,\alpha_1,\alpha_0,\alpha_1,..$ with period length $l=2$
and period
$$w_2=r_{\alpha_1}r_{\alpha_0}=\dot h_{\dot{\theta}}=2h_{\dot{\delta}}.$$
The corresponding sequence of positive roots $\tau_j$ mapped to
negative roots is
$$d^*-\dot{\theta},2d^*-\dot{\theta},3d^*-\dot{\theta},..$$
In this case
$$\lim_{n \to \infty}N^-\cap
w_n^{-1}N^+w_n=\{\left(\begin{matrix}1&x^*\\0&1\end{matrix}\right):x=\sum_{j>0}x_j
z^j\}$$ is abelian. For $g\in LSU(2)$ as in Theorem
\ref{smooththeorem},
$$\vert \sigma_0 \vert^2(g)=det(A(g)^*A(g))=\prod(1+\vert
\eta_j \vert^2)^{-j}exp(-2\sum k\vert \chi_k\vert^2)\prod(1+\vert
\zeta_i\vert^2)^{-i}$$ (see (c) of Theorem 7 of \cite{Pi1}).

\subsection{SU(3)}

In this case $h_{\dot{\delta}}$ is in the coroot lattice. For this
choice of period, there are two possible affine periodic reduced
sequences, as in Theorem \ref{periodic} (this is obvious from the
picture on page 72 of \cite{PS}). One choice corresponds to the
periodic sequence of simple roots
$\alpha_0,\alpha_1,\alpha_2,\alpha_1,..$, with period length $l=4$
(the other interchanges $1$ and $2$). The sequence of positive
roots mapped to negative roots is
$$d^*-\dot{\theta},d^*-\dot{\alpha_2},2d^*-\dot{\theta},
d^*-\dot{\alpha_1},3d^*-\dot{\theta},2d^*-\dot{\alpha_2},
4d^*-\dot{\theta},2d^*-\dot{\alpha_1},..$$  For $g\in LSU(3)$ as
in Theorem \ref{smooththeorem},
$$det\vert A(g) \vert^2=  \prod(1+\vert
\eta_j \vert^2)^{-k(j)}exp(-3\sum k\vert
\chi_k\vert^2)\prod(1+\vert \zeta_i\vert^2)^{-k(i)}$$ where $k(i)$
is the sequence $1,1,2,1,3,2,4,2,5,3,6,3,..$.

\subsection{SU(n)}

This will appear in the second author's dissertation.

\section{Appendix. The Relation between $x^*$ and $\zeta$}

In this Appendix we consider the correspondence
$$ \zeta\in c^{\infty} \leftrightarrow x^*\in c^{\infty},$$
where $l_2=exp(\sum x_j^{*}f_{\tau_j})$, as in the proof of
Theorem \ref{smooththeorem} (see (\ref{x*defn})). In the $SU(2)$
case it is known that
$$x_j^*=x_1^*(\zeta_j,\zeta_{j+1},..)
=\zeta_j\prod_{k=j+1}^{\infty}(1+\vert\zeta_k\vert^2)+R(\zeta_{j+1},..).$$
This triangular relation between $\zeta$ and $x^*$ in principle
(and in practice on a computer) leads to a combinatorial way of
computing an expansion of $\zeta_j$ in terms of $x^*_k$, $j\le k$;
see the Appendix of \cite{Pi2}.

\begin{proposition}In general
$$x_j^*=\zeta_j\left(\prod_{k=j+1}^{\infty}a(\zeta_k)^{-\tau_k(h_{\tau_j})}\right)
+R_j(\zeta_{j+1},\zeta_{j+2},..)$$
and
$$x_{nl+s}=x_s(\zeta_{nl+s},\zeta_{nl+s+1},..)$$
\end{proposition}

\begin{proof}We need to refine
(\ref{induction}) and the subsequent argument. In place of
(\ref{induction}) we assume that $k^{(n-1)}$ has a triangular
factorization with
\begin{equation}\label{newinduction}l(k^{(n-1)})=exp(\sum_{j=1}^{n-1}x_j^{*(n-1)}
f_{\tau_{j}}) \in N^-\cap w_{n-1}^{-1}N^+w_{n-1},\end{equation}
where $x_{n-1}^{*(n-1)}=\zeta_{n-1}$ and
\begin{equation}\label{xform}
x_j^{*(n-1)}=\zeta_j\prod_{k=j+1}^{n-1}a(\zeta_k)^{-\tau_k(h_{\tau_j})}
+R_j^{(n-1)}(\zeta_{j+1},\zeta_{j+2},..)\end{equation} We have
established this for $n-1=2$; see (\ref{2ndcase}). Mimicking the
calculations following (\ref{induction}), we find

$$k^{(n)}=exp(\zeta_n f_{\tau_n})a(\zeta_n)^{h_{\tau_n}}\mathbf
w_{n-1}^{-1}
exp(-\bar{\zeta}_ne_{\gamma_n})exp(\sum_{j=1}^{n-1}x_j^{*(n-1)}
e_{-w_{n-1}\tau_{j}}) \mathbf w_{n-1}a(k^{(n-1)})u(k^{(n-1)})$$
\begin{equation}\label{k3} =exp(\zeta_n
f_{\tau_n})a(\zeta_n)^{h_{\tau_n}}\mathbf w_{n-1}^{-1}
exp(\sum_{j=1}^{n-1}x_j^{*(n-1)}
exp(-ad(\bar{\zeta}_ne_{\gamma_n}))(e_{-w_{n-1}\tau_{j}}))
exp(-\bar{\zeta}_ne_{\gamma_n}) \mathbf
w_{n-1}a(k^{(n-1)})u(k^{(n-1)})\end{equation} (Note that
$-w_{n-1}\tau_j$ is a positive root, $1\le j<n$)

\begin{lemma}For $p>0$ and $j<k<n$,
$$p\gamma_n-w_{n-1}\cdot \tau_j \ne -w_{n-1}\cdot \tau_k.$$
\end{lemma}

\begin{proof} Suppose otherwise, and apply $w_{n-1}^{-1}$ to both
sides. Then $p\tau_n-\tau_j=-\tau_k$ or $\tau_j=\tau_k+p\tau_n$.
This implies
$$\gamma_j=r_j..r_{k-1}\gamma_k+pr_j..,r_{n-1}\gamma_n.$$ This is
a contradiction, because the term on the left is simple, and the
terms on the right hand side are positive by (b) of Proposition
\ref{rootlemma}.
\end{proof}

We apply this to (\ref{k3}) to obtain
\begin{equation}\label{k4}k^{(n)}=exp(\zeta_n
f_{\tau_n})a(\zeta_n)^{h_{\tau_n}}\mathbf w_{n-1}^{-1}
exp(\sum_{j=1}^{n}X_j^{*(n)} e_{-w_{n-1}\tau_{j}}+Y)
exp(-\bar{\zeta}_ne_{\gamma_n}) \mathbf
w_{n-1}a(k^{(n-1)})u(k^{(n-1)}),\end{equation} where $X_j^{*(n)}$
has the form
\begin{equation}\label{Xform}
X_j^{*(n-1)}=\zeta_j\prod_{k=j+1}^{n-1}a(\zeta_k)^{-\tau_k(h_{\tau_j})}
+R_j^{(n-1)}(\zeta_{j+1},\zeta_{j+2},..)\end{equation} and $Y\in
\mathfrak n^+\cap w_{n-1}\mathfrak n^+ w_{n-1}^{-1}$ Now we need
to write
\begin{equation}\label{decomp1}exp(\sum_{j=1}^{n}X_j^{*(n)}
e_{-w_{n-1}\tau_{j}}+Y)=exp(\sum_{j=1}^{n}\tilde x_j^{*(n)}
e_{-w_{n-1}\tau_{j}})exp(y),\end{equation} relative to the
decomposition
$$N^+=\left(N^+\cap w_{n-1}N^-w_{n-1}^{-1}\right)\left(N^+\cap
w_{n-1}N^+w_{n-1}^{-1}\right).$$ Because in general the second
group is not normal, it is probably not true in general that
$\tilde x^*= X^*$. So it is not presently clear that $\tilde
x_j^{*(n)}$ has the form (\ref{Xform}), and we need to look more
carefully at (\ref{decomp1}).

\begin{lemma}For $p>0$, $j<k<n$ and $0<L<n$, if the
$p\gamma_n-w_{n-1}\tau_L$ root space belongs to $\mathfrak n^+\cap
w_{n-1}\mathfrak n^+ w_{n-1}^{-1}$, then
\begin{equation}\label{wrong}p\gamma_n-w_{n-1}\tau_L-w_{n-1}\tau_j\ne -w_{n-1}\tau_k.\end{equation}
\end{lemma}

\begin{proof}The assumption concerning $p\gamma_n-w_{n-1}\tau_L$ is equivalent to
$p\tau_n>\tau_L$. Suppose (\ref{wrong}) is false. Applying
$w_{n-1}^{-1}$ to both sides, we obtain
$p\tau_n-\tau_L-\tau_j=-\tau_k$, or
$w_{j-1}(p\tau_n-\tau_L)+w_{j-1}\tau_k=w_{j-1}\tau_j=\gamma_j$.
The right hand side is simple. The first term on the left hand
side must be positive. Otherwise $p\tau_n-\tau_L=\tau_r$, for some
$r=1,..,j-1$. But then $p\tau_n=\tau_L+\tau_r$, which is
impossible. The second term is positive by (b) of Lemma
\ref{rootlemma}). We thus obtain a contradiction.
\end{proof}

Let $\mathfrak n_2$ denote the subalgebra of $\mathfrak n^+\cap
w_{n-1}\mathfrak n^+ w_{n-1}^{-1}$ generated by the roots
$p\gamma_n-w_{n-1}\tau_L$ as in the proceeding Lemma. The Lemma
implies that \begin{equation}\label{filtration}\mathfrak n^+\cap
w_{k}\mathfrak n^-w_k^{-1}+\mathfrak n_2 =\sum_{j\le k}\mathbb C
e_{-w_{n-1}\tau_j} +\mathfrak n_2\end{equation} is a subalgebra.

We rewrite (\ref{decomp1}) as
$$exp(\sum_{j=1}^{n}\tilde x_j^{*(n)} e_{-w_{n-1}\tau_{j}})=
exp(\sum_{j=1}^{n}X_j^{*(n)} e_{-w_{n-1}\tau_{j}}+Y)exp(-y)$$
where $Y,y\in \mathfrak n_2$. (\ref{filtration}) implies $\tilde
x_k^{*(n)}$ is a combination of $X_j^{*(n)}$, $j\le k$, $Y$, and
$y$. Since the $X_j^{*(n)}$ have the form (\ref{Xform}), the same
is true of $\tilde x_k^{*(n)}$.

Inserting this into (\ref{k3}) leads directly to
(\ref{newinduction}) with $n$ in place of $n-1$. This completes
the proof of the first statement in the Proposition. The second
statement follows from affine periodicity (see (\ref{affineeqn})).
\end{proof}

\end{document}